\documentclass[12pt,3p]{amsart}
\usepackage{bbm}
\usepackage{amsmath,amssymb, amsthm}
\usepackage[colorlinks,urlcolor=red]{hyperref}
\usepackage{graphicx, enumerate}
\theoremstyle{plain}
\newtheorem{theorem}{Theorem}[section]

\newtheorem{lemma}[theorem]{Lemma}

\numberwithin{equation}{section}\theoremstyle{definition}

\theoremstyle{remark}

\allowdisplaybreaks

\numberwithin{equation}{section}

\newcommand{\R}{\Bbb R}

\title[On the non-existence of \protect{$srg(76,21,2,7)$}]
{On the non-existence of \protect{$\mathbf{srg(76,21,2,7)}$}}
\author[M.R.~Alfuraidan, I.O.~Sarumi, S.~Shpectorov]{Monther R. Alfuraidan$^{1}$, Ibrahim O. Sarumi$^{2}$, Sergey Shpectorov$^{3}$}

\begin{document}
\maketitle

\vspace*{-0.5cm}


\begin{center}
{\footnotesize  $^{(1,2)}$Department of Mathematics \& Statistics,
King Fahd University of Petroleum and Minerals\\ Dhahran 31261, Saudi Arabia\\
monther@kfupm.edu.sa, g201304910@kfupm.edu.sa,    \\$^{(3)}$School of Mathematics,
University of Birmingham, Birmingham, United Kingdom \\ s.shpectorov@bham.ac.uk}
\end{center}

\hrulefill

{\footnotesize \noindent {\bf Abstract.}
We present a new non-existence proof for the strongly regular graph $G$ with 
parameters $(76,21,2,7)$, using the unit vector representation of the graph.

\vskip 0.5cm
\noindent {\bf Keywords}: Strongly regular graph, distance regular graph, unit vector representation.

\noindent {\bf AMS Subject Classification}: Primary: 05E30, Secondary: 	05C30 }

\hrulefill
\section{Introduction}

A graph $G$ is said to be strongly regular with parameters $(v,k,\lambda,\mu)$ if 
the following condition holds: $G$ has $v$ vertices (i.e., $|V(G)|=v$) and, for 
$u,w\in V(G)$, the number of common neighbours of $u$ and $w$ in $G$ is $k$ if 
$u=w$ (so $G$ is regular of valency $k$), $\lambda$ if $u$ and $w$ are adjacent, 
and $\mu$ if $u$ and $w$ are non-adjacent. Strongly regular graphs are among the 
central objects in graph theory and its applications. We write 
$srg(v,k,\lambda,\mu)$ for any strongly regular graph with parameters 
$(v,k,\lambda,\mu)$.
 
Haemers \cite{WH} proved non-existence of $srg(76,21,2,7)$. His proof is very 
efficient (perhaps even a bit terse), and it relies on clever edge counting 
to establish that such $G$ must locally be a union of $3$-cliques. This means 
that $G$ is the collinearity graph of a point-line geometry $pg(3,6,1)$ (a 
generalized quadrangle of order $(3,6)$). At this point Haemers quotes the 
non-existence result for $pg(3,6,1)$ by Dixmier and Zara \cite{SF}.

In this note, we give an alternative proof of Haemers' theorem based on the 
well-known fact that every distance regular graph (and in particular, every 
strongly regular graph) admits a Euclidean realization as a set of unit vectors 
in an eigenspace of the adjacency matrix of $G$. In this realization, the value 
of the inner product of two vectors (the cosine of the angle between them) is 
fully determined by the mutual distance of the corresponding vertices. This is 
encoded in the so-called cosine sequence. Note that the eigenvalues of the 
adjacency matrix, dimension of each eigenspace, and the cosine sequence can be 
easily deduced from the parameters of $G$ via the readily available formulas 
(for example, see \cite{CG}).

There are many open cases of strongly regular graphs even for relatively small 
values of $v$ (see the table of feasible parameters up to $v=100$ in \cite{BH}). 
Of course, the aim of our project is to contribute to one of the open cases. 
In this sense, the proof in this note is just a sample of things to come. 
However, we think that even this taster proof demonstrates efficiency of the 
method and it exhibits interesting features, such as the relation to roots 
systems, which arise in our proof not once, but twice.

Just like Haemers, we aim to show that $G$ is locally a union of cliques. However, 
once we achive this, we do not stop, but rather use our unit vector setup to 
achieve an outright contradiction. In this sense, we also provide an alternative 
proof of the result of Dixmier and Zara. 

\section{Starting point}

Suppose $G$ is $srg(76,21,2,7)$. Then the adjacency matrix of $G$ has eigenvalues 
$21$, $2$, and $-7$ with multiplicities $1$, $56$, and $19$, respectively. We focus 
on the $19$-dimensional eigenspace corresponding to the eigenvalue $-7$. The cosine 
sequence for this eigenspace is $(1,-\frac{1}{3},\frac{1}{9})$. This means
that our graph $G$ can be realized as a set of $76$ unit vectors $x_v$, $v\in V(G)$,
in the Euclidean space $\R^{19}$ such that $(x_u,x_v)=-\frac{1}{3}$ if the distinct
vertices $u$ and $v$ are adjacent, and $(x_u,x_v)=\frac{1}{9}$, if they are not.

From now on we identify vertices of $G$ with the corresponding unit vectors.
Hence we simply write $u$ and $v$ in place of $x_u$ and $x_v$.

\section{Neighbourhood}

Fix an arbitrary $u\in V(G)$. The subgraph induced on the $21$ vertices in $G_1(u)$
is a union of cycles $C_1,\ldots,C_r$, since its degree $\lambda$ is $2$. Let us
slightly alter vectors $v\in G_1(u)$ to make them perpendicular to $u$. Namely, we
set $\hat v:=\frac{3}{2}v+\frac{1}{2}u$ for each $v\in G_1(u)$. Clearly, $(u,\hat v)=
\frac{3}{2}(u,v)+\frac{1}{2}(u,u)=\frac{3}{2}(-\frac{1}{3})+\frac{1}{2}1=0$, as
desired. Also, for $v,w\in G_1(u)$, we have $(\hat v,\hat w)=\frac{9}{4}(v,w)+
\frac{3}{4}(u,w)+\frac{3}{4}(v,u)+\frac{1}{4}(u,u)=\frac{9}{4}(v,w)-\frac{1}{4}$.
Therefore,
$$
(\hat v,\hat w)=\left\{
  \begin{array}{r l}
    2,  & \mbox{if $v=w$,}\\
    -1, & \mbox{if $v$ and $w$ are adjacent,}\\
    0,  & \mbox{if they are not adjacent.}
  \end{array} \right.
$$

Let $V_i:=\langle\hat v\mid v\in V(C_i)\rangle$ be the subspace of $\R^{19}$ spanned
by the vectors corresponding to the vertices of the $i$th cycle $C_i$ in $G_1(u)$. It
follows from the above inner product values that $u\perp V_i$ for all $i$ and that
$V_i\perp V_j$ for all $i\neq j$.

\begin{lemma} \label{dimension}
Let $V(C_i)=\{v_1,v_2,\ldots,v_t\}$. Then we have:
\begin{enumerate}
\item[(i)] $\hat v_1+\hat v_2+\cdots+\hat v_t=0$; and
\item[(ii)] $\dim V_i=t-1$.
\end{enumerate}
\end{lemma}

\begin{proof} (i) Let $\hat v=\hat v_1+\hat v_2+\cdots+\hat v_t$. Then, for each $j$,
we have that $(\hat v,\hat v_j)=0$, since $\hat v_j$ itself contributes $2$ to the sum,
and its two neighbours contribute $-1$ each, while all the other vertices of $C_i$
contribute naught. Therefore, $(\hat v,\hat v)=\sum_{j=1}^t(\hat v,\hat v_j)=0$,
proving that $\hat v=0$.

\medskip\noindent
(ii) Assuming that the vertices $v_1,v_2,\ldots,v_t$ appear in this order on the cycle
$C_i$, let $A_{t-1}$ be the Gram matrix of the vectors $\hat v_1,\hat v_2,\ldots,\hat v_{t-1}$.
Then
$$
A_{t-1} = \left(
\begin{array}{rrrrrr}
2  & -1  & 0  & 0  & .  & 0  \\
-1 &  2  & -1 & 0  & .  & 0  \\
0  & -1  & 2  & -1 & .  & 0  \\
.  & .   & .  & .  & .  & .  \\
0  & .   & .  & -1 & 2  & -1 \\
0  & .   & .  & .  & -1 & 2
\end{array}
\right).
$$
Let $d_{t-1}$ be the determinant of $A_{t-1}$. Viewing $r=t-1$ as variable, we
obtain the recursive relation $d_r=2d_{r-1}-d_{r-2}$ by expanding the determinant
along the bottom row. Taking into account that $d_1=2$ and $d_2=3$, we easily
deduce that $d_r=r+1\neq 0$, and so $\hat v_1,\hat v_2,\ldots,\hat v_{t-1}$ are
linearly independent.
\end{proof}

We included this proof for completeness; however, we need to mention that these
facts are well known. Indeed, the matrix above is the Gram matrix of a basis from
the root system of type $A_{t-1}$, and if we add the missing vector $\hat v_t$
then this gives the basis of the affine root system $\tilde A_{t-1}$.

We now focus on a vertex $w$ from $G_2(u)$ and study the $\mu=7$ neighbours of
$w$ in $G_1(u)$. Let $s_i$ be the number of such neighbours on the cycle $C_i$.

\begin{lemma} \label{three}
The length $t_i$ of $C_i$ is a multiple of $3$; namely, $t_i=3s_i$.
\end{lemma}

\begin{proof}
Let again $V(C_i)=\{v_1,v_2,\ldots,v_t\}$, where $t=t_i$, and $\hat v=
\hat v_1+\hat v_2+\cdots+\hat v_t$.

Note that $(\hat v_j,w)=(\frac{3}{2}v_j+\frac{1}{2}u,w)=\frac{3}{2}(v_j,w)+\frac{1}{2}(u,w)$.
If $w$ is adjacent to $v_j$, this results in $-\frac{1}{2}+\frac{1}{18}=
-\frac{4}{9}$, and otherwise, the result is $\frac{1}{6}+\frac{1}{18}=
\frac{2}{9}$. Now consider the equality
$$0=(0,w)=(\hat v,w)=(\hat v_1,w)+(\hat v_2,v)+\cdots+(\hat v_t,x).$$
Since $w$ is adjacent to $s=s_i$ vertices and non-adjacent to $t-s$
vertices, we obtain from here that
$$0=-\frac{4}{9}s+\frac{2}{9}(t-s),$$
which gives $t=3s$, as claimed.
\end{proof}

\section{Second layer}

We alter the vertices in $G_2(u)$ in a similar way to make them perpendicular
to $u$. For $w\in G_2(u)$, we set $\hat w:=\frac{9}{4}w-\frac{1}{4}u$.
Then $(\hat w,u)=\frac{9}{4}\frac{1}{9}-\frac{1}{4}1=0$, as claimed. Similarly,
we compute, for $v,w\in G_2(u)$,
$$
(\hat v,\hat w)=
\left\{
  \begin{array}{r l}
    5,  & \mbox{if $v=w$,}\\
   -\frac{7}{4}, & \mbox{if $v$ adjacent to $w$,}\\
    \frac{1}{2}, & \mbox{if $v$  is not adjacent to $w$.}
  \end{array}\right.
$$
Finally, we also compute, and also in a very similar way, the inner products
$(\hat v,\hat w)$ for $v\in G_1(u)$ and $w\in G_2(u)$. These are:
$$
(\hat v,\hat w)=
\left\{
  \begin{array}{r l}
   -1, & \mbox{if  $v$ is adjacent to $w$,}\\
   \frac{1}{2}, & \mbox{if $v$ is not adjacent to $w$.}
  \end{array}
\right.
$$

Recall that every vertex $w\in G_2(u)$ has seven neighbours in $G_1(u)$. Let
us first describe the subgraph $M=M_w$ induced on these seven vertices.

\begin{lemma}
Each connected component of $M$ is of size $1$ or $2$.
\end{lemma}

\begin{proof}
If $xyz$ is a $2$-path in $M$, with $x\neq z$, then $uxwz$ is a $4$-cycle
in $G_1(y)$, a contradiction with Lemma \ref{three} with $y$ in place of $u$.
\end{proof}

If $x$ is a size $1$ component of $M$ then the projection $p_x$ of $\hat w$
to the $1$-space spanned by $\hat x$ coincides with
$\frac{(\hat w,\hat x)}{(\hat x,\hat x)}\hat x=-\frac{1}{2}\hat x$. Hence
$(p_x,p_x)=\frac{1}{4}(\hat x,\hat x)=\frac{1}{2}$. If $xy$ is a size $2$
component of $M$ then by symmetry the projection $p_{xy}$ of $\hat w$ to
the subspace spanned by $\hat x$ and $\hat y$ is a multiple of
$d=\hat x+\hat y$. Note that $(d,d)=(\hat x+\hat y,\hat x+\hat y)=
2-1-1+2=2$ and $(\hat w,d)=(\hat w,\hat x+\hat y)=-1-1=-2$. Hence
$p_{xy}=\frac{(\hat w,d)}{(d,d)}d=-d$, and so $(p_{xy},p_{xy})=(-d,-d)=
(d,d)=2$.

Projections corresponding to different components of $M$ are orthogonal.
Hence, if we have $k$ components of size $2$ and, correspondingly, $7-2k$
components of size $1$ then the length of the projection of $\hat w$ to
the subspace of $V$ spanned by all $\hat x$, $x\in M$, is
$2k+\frac{1}{2}(7-2k)=2k-k+\frac{7}{2}=k+\frac{7}{2}$. Since this must be at
most $(\hat w,\hat w)=5$, we conclude that $k=0$ or $1$.

Consider one of the cycles $C=C_i$ of length $t=t_i$ and $N=C\cap M$
consisting of $s=s_i$ vertices. We know that $t=3s$. Let $U=V_i$, the subspace
spanned by the vectors $\hat v$, $v\in C$, and let $p$ be the projection of
$\hat w$ onto $U$.

\begin{lemma} \label{lowest}
We have $(p,p)\geq\frac{s}{2}$. Furthermore, $(p,p)=\frac{s}{2}$ if and only
the vertices of $N$ are evenly spaced in $C$, containing every third vertex
along $C$ and $p=-\frac{1}{2}\sum_{v\in N}\hat v$.
\end{lemma}

\begin{proof}
Let $p'$ be the projection of $\hat w$ to the subspace spanned by $\hat v$ for
$v\in N$. Then certainly $(p,p)\geq (p',p')$ with equality holding only if
$p=p'$. The graph $N$ consists of several components of $M$. The computation
before the lemma shows that a component of size $1$ contributes $\frac{1}{2}$
to $(p',p')$ and a component of size $2$ contributes $2>2\cdot\frac{1}{2}$.
Hence $(p,p)\geq (p',p')\geq s\cdot\frac{1}{2}=\frac{s}{2}$, as claimed.
Furthermore, the equality only holds when every component is of size $1$ and
also $p=p'=\sum_{v\in N}p_v=-\frac{1}{2}\sum_{v\in N}\hat v$. In particular,
$N$ is an independent subset of $C$. Taking now a vertex $x\in C\setminus N$,
we see that $\frac{1}{2}=(\hat w,\hat x)=(p,\hat x)=(p',\hat x)=
-\frac{1}{2}(\sum_{v\in N}\hat v,\hat x)=-\frac{1}{2}(-m)$, where $m$ is the
numbers of vertices of $N$ adjacent to $x$. This shows that $m=1$ for every
$x\in C\setminus N$, and hence $N$ is evenly spaced in $C$.

Conversely, if $N$ is evenly spaced then $(p',\hat x)=\frac{1}{2}=
(\hat w,\hat x)$ for every $x\in C\setminus N$ and $(p',\hat x)=
-\frac{1}{2}(\hat x,\hat x)=-\frac{1}{2}\cdot 2=-1=(\hat w,\hat x)$ for
every $x\in N$. Hence $p=p'$.
\end{proof}

We now assume that $k=1$ and focus on the cycle $C=C_i$ containing the
only component of $M$ of size $2$. Without loss of generality, we may assume
that $C=v_1v_2\cdots v_t$ and $v_1v_2$ is the size $2$ component of
$N=C\cap M$.

Clearly, $t\geq 6$ and it follows from Lemma \ref{dimension} that $t\leq 15$,
since $\dim V\leq 18$. Also, it follows from Lemma \ref{lowest} that the
length $(p,p)$ of the projection of $\hat w$ onto the subspace $U$ corresponding
to $C$ is at most $5-\frac{1}{2}(7-s)=\frac{s+3}{2}$.

\begin{lemma} \label{projections}
If $w$ is adjacent to $v_1$ and $v_2$ then $\hat w\in V$, $p=-(\hat v_1+
\hat v_2)+\sum_{m=1}^{s-1}\frac{1}{2}(\hat v_{3m+1}+\hat v_{3m+2})$, and
$(p,p)=\frac{s+3}{2}$. Furthermore, the subgraph $N_j=C_j\cap M$ is evenly
spaced in $C_j$ for each $j\neq i$.
\end{lemma}

\begin{proof}
The set of $t-4=3s-4$ vertices $\{v_4,v_5,\ldots,v_{t-1}\}$ consists of $s-2$
vertices from $N$ (neighbours of $w$) and $(3s-4)-(s-2)=2(s-1)$ other vertices
(non-neighbours). Let us view the $s-2$ neighbours as dividers splitting the
non-neighbours into $s-1$ continuous parts $R_j$ of lengths
$r_1,r_2,\ldots,r_{s-1}$. Let $d=\hat v_1+\hat v_2$ and, for $j=1,\ldots,s-1$,
we set $d_j=\sum_{v\in R_j}\hat v$. We let $U'$ be the subspace of $U$
spanned by $d,d_1,\ldots,d_{s-1}$ and let $p'$ be the projection of $\hat w$
onto $U'$. Then, clearly, $(p,p)\geq(p',p')$ and, since the vectors
$d,d_1,\ldots,d_{s-1}$ are pairwise orthogonal and of length $2$, we have that
$p'=-d+\frac{1}{4}\sum_{j=1}^{s-1}r_j d_j$, which means that $(p',p')=
2+\frac{1}{8}\sum_{j=1}^{s-1}r_j^2$. Hence, to find the minimum of the latter,
we need to minimize $\sum_{j=1}^{s-1}r_j^2$ under the restriction that
$\sum_{j=1}^{s-1}r_j=2(s-1)$. Clearly, the minimum is achieved when all
$r_j$ are equal, that is when all $r_j$ are equal to $\frac{2(s-1)}{s-1}=2$.
The minimum value $(p',p')$ is, therefore, $2+\frac{1}{8}(s-1)2^2=2+\frac{s-1}{2}=
\frac{s+3}{2}$.

Hence $\frac{s+3}{2}\geq(p,p)=(p',p')\geq\frac{s+3}{2}$. Clearly, this means that
$p=p'$ is of length $\frac{s+3}{2}$, and so every part $R_j$ is of size $2$, which
leads to the vectors in the statement of the lemma. Also for every cycle $C_m$ other
than $C$ we must have the minimum length value $\frac{s_m}{2}$ and so the vertices
$C_m\cap M$ must be evenly spaced in $C_m$.
\end{proof}

Let us adopt the following terminology: the vectors $d=\hat v_i+\hat v_{i+1}$ will be
called \emph{pairs}, while the edge $v_iv_{i+1}$ will be called the \emph{base} of
the pair $d$. Using these terms, $p$ in the lemma above is the sum of the unique
\emph{minus-pair} $-(\hat v_1+\hat v_2)$ and $s-1$ \emph{half-pairs}
$\frac{1}{2}(\hat v_{3m+1}+\hat v_{3m+2})$.

\begin{lemma} \label{cliques}
Every cycle $C_i$ in $G_1(u)$ has length $3$.
\end{lemma}

\begin{proof}
Suppose, by contradiction, that $C=C_i=v_1v_2\ldots v_t$ has length $t\geq 6$.
In $G_1(v_1)$, $u$ is adjacent to $v_t$ and $v_2$, which are not adjacent to
each other. Hence $v_tuv_2$ is part of a cycle $D$ in $G_1(v_1)$ of length at
least $6$. Let $w\neq u$ be the second neighbour of $v_2$ in $D$, and let
$w'\neq v_2$ be the second neighbour of $w$ in $D$. Note that $w$ is adjacent
to $v_1$ and $v_2$, and hence $\hat w$ is as in Lemma \ref{projections}.
In particular, $p=-(\hat v_1+\hat v_2)+\sum_{m=1}^{s-1}\frac{1}{2}(\hat v_{3m+1}+
\hat v_{3m+2})$ is the projection of $\hat w$ to the subspace $U=V_i$.

We obtain a contradiction by computing $(\hat w,\hat w')$. Since $w$ and $w'$
are adjacent vertices in $G_2(u)$ (note that $w$ and $w'$ are not adjacent to
$u$, since $D$ has length at least $6$), the value of the inner product must be
$-\frac{7}{4}$. On the other hand, we can estimate the value as follows.
Recall that $\hat w\in V$ by Lemma \ref{projections} and so $\hat w=\sum_{j=1}^r p_j$,
where $p_j$ is the projection of $\hat w$ to the subspace $V_j$ corresponding
to the cycle $C_j$ in $G_1(u)$. We already know $p=p_i$ and, by Lemmas
\ref{projections} and \ref{lowest}, if $j\neq i$ then $p_j=
-\frac{1}{2}\sum_{v\in N_j}\hat v$, where $N_j=M\cap C_j$ is evenly spaced in $C_j$.
It follows that $(\hat w,\hat w')=\sum_{j=1}^r(p_j,\hat w)$. Clearly, $w'$ is adjacent
to $v_1$ but not to $v_2$. Hence $(-(\hat v_1+\hat v_2),\hat w')=-(-1+\frac{1}{2})=
\frac{1}{2}$. Consider now a half-pair $\frac{1}{2}(\hat v_{3m+1}+\hat v_{3m+2})$.
If $w'$ is adjacent to both $v_{3m+1}$ and $v_{3m+2}$ then $\hat w'$ is described
as in Lemma \ref{projections} with the minus-pair base $-$ $v_{3m+1}v_{3m+2}$. This
means, however, that $v_1v_2$ is the base of a half-pair for $w'$. Hence $w'$
cannot be adjacent to $v_1$, a contradiction. Therefore, $w'$ is adjacent to at most
one of $v_{3m+1}$ and $v_{3m+2}$. If $w'$ is adjacent to one of these then
$(\frac{1}{2}(\hat v_{3m+1}+\hat v_{3m+2}),\hat w')=\frac{1}{2}(-1+\frac{1}{2})=
-\frac{1}{4}$. If $w'$ is adjacent to neither of them then
$(\frac{1}{2}(\hat v_{3m+1}+\hat v_{3m+2}),\hat w')=\frac{1}{2}(\frac{1}{2}+\frac{1}{2})=
\frac{1}{2}$. Hence the smallest possible value of $(p,\hat w')$ is
$\frac{1}{2}+(s-1)(-\frac{1}{4})=\frac{3}{4}-\frac{s}{4}$. In all $C_j\neq C$, $w$ is adjacent
to $7-s$ vertices $v$, and for each such vertex, $\hat v$ appears in $\hat w$ with coefficient
$-\frac{1}{2}$. If $w'$ is adjacent to $v$ then this gives contribution $-\frac{1}{2}(-1)=
\frac{1}{2}$ to the value of $(\hat w,\hat w')$. If $w'$ and $v$ are not adjacent then
the contribution is $-\frac{1}{2}\frac{1}{2}=-\frac{1}{4}$. Hence the smallest possible
contribution from all vectors $\hat v$ appearing in $\hat w$, where $v\not\in C$,
is $(7-s)(-\frac{1}{4})=-\frac{7}{4}+\frac{s}{4}$. Putting all of the above together, we
conclude that $(\hat w,\hat w')\geq (\frac{3}{4}-\frac{s}{4})+(-\frac{7}{4}+\frac{s}{4})=
-1$. This clearly is a contradiction since $(\hat w,\hat w')=-\frac{7}{4}$.
\end{proof}

\section{Contradiction}

Vertices and $4$-cliques of $G$ form a point-line geometry. It follows from Lemma 
\ref{cliques} that every point lies in seven lines and then, using the parameters of 
$G$, it is easy to deduce that this geometry is a generalized quadrangle of order 
$(3,6)$, which cannot exist due to a theorem of Dixmier and Zara \cite{SF}. However,
with the wealth of information that we have collected, we can achieve a quick 
contradiction not using \cite{SF}.

Let $T=S^\perp$, where $S$ is the span of the vectors in $\{u\}\cup G_1(u)$. That
is, $S=\langle u\rangle\oplus V$. Since all cycles in $G_1(u)$ are of length $3$,
Lemma \ref{dimension} shows that $\dim S=1+7\cdot 2=15$, and so $\dim T=4$.

If $w\in G_2(u)$ then the projection of $\hat w$ onto $V$ coincides with $-\frac{1}{2}\sum_{v\in M}\hat v$ and it has length $\frac{7}{2}$. It follows that
the projection of $\hat w$ onto $T$ has length $5-\frac{7}{2}=\frac{3}{2}$. Let
$w^\circ$ denote $\frac{2}{\sqrt{3}}$ times the projection of $\hat w$ onto $T$.
Then $(w^\circ,w^\circ)=2$. We now compute $(w^\circ,(w')^\circ)$ for distinct
$w,w'\in G_2(u)$.

If $w$ and $w'$ are adjacent then $(\hat w,\hat w')=-\frac{7}{4}$. Note that
the edge $ww'$ lies in a unique $4$-clique and so $w$ and $w'$ have a unique
common neighbour in $G_1(u)$. It follows that if $p$ and $p'$ are the projections
of $\hat w$ and $\hat w'$ onto $V$ then $(p,p')=\frac{1}{2}+6(-\frac{1}{4})=-1$.
Hence $(w^\circ,(w')^\circ)=\frac{4}{3}(-\frac{7}{4}+1)=-1$.

If $w$ and $w'$ are non adjacent then $(\hat w,\hat w')=\frac{1}{2}$. Let $\beta$
the the number of common neighbours of $w$ and $w'$ in $G_1(u)$. Then
$(p,p')=\beta\frac{1}{2}+(7-\beta)(-\frac{1}{4})=-\frac{7}{4}+\frac{3\beta}{4}$
and $(w^\circ,(w')^\circ)=\frac{4}{3}(\frac{1}{2}-(-\frac{7}{4}+\frac{3\beta}{4}))=
3-\beta$.

To summarize, if $w,w'\in G_2(u)$ and $\beta=|G_1(u)\cap G_1(w)\cap G_1(w')|$ then
$$
(w^\circ,(w')^\circ)=
\left\{
\begin{array}{l l}
2,  & \mbox{if $w=w'$,}\\
-1, & \mbox{if $w$ and $w'$ are adjacent,}\\
3-\beta, & \mbox{if and $w$ and $w'$ are non-adjacent.}
\end{array}\right.
$$
Clearly, it follows that $1\leq\beta\leq 5$.

Notice that the above values of inner products mean that all vectors $w^\circ$,
$w\in G_2(u)$ are contained in a root system in $T$. Indeed, since all values
are integers, the vectors $w^\circ$ span an integral lattice and all vectors
of length $2$ from that lattice form a root system.

The largest root system in dimension $4$ is $D_4$ having $24$ vectors splitting
into $12$ pairs of opposite roots. Since $|G_2(u)|=54>4\cdot 12$, we must have
five vertices $\{w_1,\ldots,w_5\}$ such that all vectors $(w_i)^\circ$ belong
to the same pair of opposite roots.

\begin{lemma}
There is no strongly regular graph with parameters $(76,21,2,7)$.
\end{lemma}

\begin{proof}
Consider the five vertices $\{w_1,\ldots,w_5\}$ such that all vectors $(w_i)^\circ$
are in the same pair of opposite roots $\{r,-r\}$. Without loss of generality, let
the first $s\geq 3$ of the vectors $(w_i)^\circ$ be $r$ and the remaining $5-s$ be
$-r$.

From the calculations above, the vertices $w_i$ are pairwise non-adjacent.
Furthermore, if $(w_i)^\circ=(w_j)^\circ$ then $w$ and $w'$ have exactly one common
neighbour in $G_1(u)$, and if $(w_i)^\circ=-(w_j)^\circ$ then $w_i$ and $w_j$
have exactly five common neighbours in $G_1(u)$.

If $s\neq 5$ then $(w_5)^\circ=-r$ and so both $w_1$ and $w_2$ must have five
neighbours among the seven vertices from $M=G_1(u)\cap G_1(w_5)$. However, this
means that $w_1$ and $w_2$ have at least two common neighbours in $M$; a
contradiction. Therefore, $s=5$ and any two vectors $w_i$ and $w_j$ share
a unique common neighbour in $G_1(u)$.

For the final contradiction, note that there are at most three $3$-cycles
in $G_1(u)$ where $w_1$, $w_2$, and $w_3$ may have common neighbours. It follows
that there are at least four $3$-cycles $C$ where $w_1$, $w_2$ and $w_3$ are
adjacent to the three distinct vertices of $C$. This means that, in each of
these $3$-cycles $C$, the vertex $w_4$ would have the same neighbour as one
of the vertices $w_1$, $w_2$, and $w_3$. Clearly, this means that $w_4$ must
share at least two common neighbours with one of the vectors $w_1$, $w_2$, or
$w_3$; a contradiction.
\end{proof}

\bigskip
\noindent {\large{\textbf{Acknowledgement}}}

\medskip\noindent 
The authors are grateful to King Fahd University of Petroleum and Minerals for supporting this research.

\bibliographystyle{abbrv}

\end{document}